\documentclass[12pt]{article}

\usepackage[all]{xy}
\usepackage{xypic,amsmath,amsfonts,amsthm,amssymb,amstext,amsbsy,graphicx,comment,epsfig}
\usepackage{charsub}
\usepackage{mathtools}
\usepackage{color, colortbl}
\usepackage{multicol}
\usepackage{graphicx}
\usepackage{lscape}
\definecolor{Gray}{gray}{0.9}
\begin{document}
\selectfont
\define\lG{\lambda(G)}
\define\lg{\lambda(g)}
\define\Gal{Gal}
\define\lgi{\lambda(g)^{-1}}
\def\on#1{\begin{equation*}#1\end{equation*}}
\def\al#1{\begin{align*}#1\end{align*}}
\def\gen#1{\langle #1\rangle }
\define\rat{\mathbb{Q}}
\define\Fq{\mathbb{F}_q}
\define\Fqtoo{\mathbb{F}_{q^{p^2}}}
\define\Fqn{\mathbb{F}_{q^{{p^n}}}}

\newcommand{\verteq}{\rotatebox{90}{$\,=$}}
\newcommand{\equalto}[2]{\underset{\scriptstyle\overset{\mkern4mu\verteq}{#2}}{#1}}
\define\trip[#1][#2][#3]{\underset{#1}{\underset{#2}{#3}}}
\define\br[#1]{\{#1\}}

\def\keywords#1{\def\@keywords{#1}}
\parskip=0.125in
\title{Characteristic Subgroup Lattices and Hopf-Galois Structures}
\date{}
\author{Timothy Kohl\\
Department of Mathematics and Statistics\\
Boston University\\
Boston, MA 02215\\
tkohl@math.bu.edu}
\maketitle
\begin{abstract}
The Hopf-Galois structures on normal extensions $K/k$ with $G=Gal(K/k)$ are in one-to-one correspondence with the set of regular subgroups $N\leq B=Perm(G)$ that are normalized by the left regular representation $\lambda(G)\leq B$. Each such $N$ corresponds to a Hopf algebra $H_N=(K[N])^G$ that acts on $K/k$. Such regular subgroups $N$ need not be isomorphic to $G$ but must have the same order. One can subdivide the totality of all such $N$ into collections $R(G,[M])$ which is the set of those regular $N$ normalized by $\lambda(G)$ and isomorphic to a given abstract group $M$ where $|M|=|G|$. There arises an injective correspondence between the characteristic subgroups of a given $N$ and the set of subgroups of $G$ stemming from the Galois correspondence between sub-Hopf algebras of $H_N$ and intermediate fields $k\subseteq F\subseteq K$. We utilize this correspondence to show that for certain pairings $(G,[M])$, the collection $R(G,[M])$ must be empty.
\end{abstract}
\noindent {\it key words:} Hopf-Galois extension, Greither-Pareigis theory, regular subgroup, Galois correspondence\par
\noindent {\it MSC:} 16T05,20B35,20E07\par
\renewcommand{\thefootnote}{}
\section{Introduction}
Hopf-Galois theory is a generalization of the ordinary Galois theory for fields. If one has a Galois extension of fields $K/k$ with $G=Gal(K/k)$ then the elements of $G$ act as automorphisms of course, but if one takes $k$-linear combinations of these automorphisms, one gets an injective homomorphism $\mu:k[G]\rightarrow End_k(K)$ where $\mu(\sum_{g\in G}c_g\cdot g)(a)=\sum_{g\in G}g(a)$, since a sum of automorphisms is no longer an automorphism but is an endomorphism of $K$. As such, we replace the group $G$ by the group ring, and prototype Hopf algebra, $k[G]$. Furthermore, by linear independence of characters, one has that $K\otimes k[G]\cong End_k(K)$, which means that if we augment these sums of automorphisms by left-multiplication by elements of $K$ then this yields all the $k$-endomorphisms of $K$. To be more precise, the previous isomorphism is actually $K\# k[G]\cong End_k(K)$ where $K\# k[G]$ is the so-called smash product of $K$ with $k[G]$ which, as a vector space is $K\otimes_k k[G]$ but the multiplication is as follows $(a\# h)(a' \# h')=a h(a') \# hh'$ where $a,a'\in K$ and $h,h'\in k[G]$. Moreover, if $h=\sum_{g\in G}c_g\cdot g\in k[G]$ then if $x\in k$ then 
$$
h(x)=\sum_{g\in G}c_g g(x)=(\sum_{g\in G}c_g)x
$$
namely, $h$ acts by scalar multiplication on $x$. The idea behind Hopf-Galois theory is to find a Hopf algebra which acts in a similar fashion as $k[G]$ does when the extension is Galois in the usual sense. The formal definition is as follows.
\begin{definition}
An extension $K/k$ is Hopf-Galois if there is a $k$-Hopf algebra $H$ and a $k$-algebra homomorphism $\mu:H\rightarrow End_k(K)$ such that
\begin{itemize}
\item $\mu(ab)=\sum_{(h)}\mu(h_{(1)}(a)\mu(h_{(2)})(b)$
\item $K^{H}=\{a\in K\ |\ \mu(h)(a)=\epsilon(h)a\ \forall h\in H\}=k$
\item $\mu$ induces $I\otimes\mu:K\#H\overset{\cong}\rightarrow End_k(K)$\par
\noindent where $\Delta(h)=\sum_{(h)}h_{(1)}\otimes h_{(2)}$ is the comultiplication in $H$ and $\epsilon:H\rightarrow k$ is the co-unit map.
\end{itemize}
\end{definition}
The original intended application \cite{ChaseSweedler1969} was to devise a Galois theory for purely inseparable extensions. However, it turned not to be suitable to extensions of exponent greater than $1$. However, in \cite{GreitherPareigis1987} Greither and Pareigis showed that Hopf-Galois theory can be effectively applied to separable extensions, especially those which are non-normal. As such, one obtains a 'Galois structure' on extensions such as $\mathbb{Q}(\root 3\of 2)/\mathbb{Q}$ which aren't Galois extensions in the usual sense. There are two particularly interesting features to this result, namely a given extension $K/k$ may have more than one Hopf-Galois structure on it, and also, an extension which {\it is} Galois in the usual sense (and thus Hopf-Galois with respect to the group ring $k[G]$) but also Hopf-Galois with respect to other Hopf algebra actions. It is the latter case that we are looking at here, and we give the main theorem in \cite{GreitherPareigis1987} for such extensions:
\begin{theorem}\cite{GreitherPareigis1987}
Let $K/k$ be a finite Galois extension with $G=Gal(K/k)$.
$G$ acting on itself by left translation yields an embedding
$$
\lambda:G\hookrightarrow B=Perm(G)$$
Definition: $N\leq B$ is {\it regular} if $N$ acts transitively and fixed point freely on $G$. The following are equivalent:\par
\begin{itemize}
\item There is a $k$-Hopf algebra $H$ such that $K/k$ is $H$-Galois\par
\item There is a regular subgroup $N\leq B$ s.t. $\lambda(G)\leq Norm_B(N)$ where $N$ yields $H=(K[N])^{G}$.
\end{itemize}
\end{theorem}
We note that $N$ must necessarily have the same order as $G$, but need not be isomorphic. As such, the enumeration of Hopf-Galois structures on a normal extension $K/k$ becomes a group theory problem. To organize the enumeration of the Hopf-Galois structures, one considers
$$
R(G)=\{N\leq B\ |\text{$N$ regular and $\lambda(G)\leq Norm_B(N)$}\}
$$
which are the totality of all $N$ giving rise to H-G structures, which we can subdivide into isomorphism classes given that $N$ need not be isomorphic to $G$, to wit, let
$$
R(G,[M])=\{N\in R(G)\ |\ N\cong M\}
$$
for each isomorphism class $[M]$ of group of order $|G|$. Now, the enumeration of $R(G,[M])$ for different pairings of groups of different types has been extensively studied by the author and others, e.g. \cite{Kohl1998},\cite{Byott2004pq},\cite{FCC12},\cite{CarnahanChilds1999}. One may consider the enumeration based on the different types or sizes of the groups in question, such as  $G$ cyclic, elementary abelian, $G=S_n$, $G=A_n$, $|G|=mp$, $G$ simple, $G$,$M$ nilpotent and more. What we shall consider is when $R(G,[M])=\O$.\par
The condition that $\lambda(G)\leq Norm_B(N)$ is the deciding factor as to whether a given regular subgroup $N\leq B$ gives rise to a Hopf-Galois structure. And, as such, this condition may, for $N$ of a given isomorphism type $[M]$, imply that $R(G,[M])=\O$. In some instances, basic structural properties of the groups $G$ and $N$ preclude the containment $\lambda(G)\leq Norm_B(N)$, for example in \cite{Kohl1998} it is shown that $R(C_{p^n},[M])=\O$ if $M$ is non-cyclic by comparing the exponent of $C_{p^n}$ versus that of the Sylow $p$-subgroup of $Norm_B(N)$ when $N$ is a $p$-group. For other cases, some deeper analysis is needed. In \cite{Byott2004simple}, Byott proved that if $G$ is simple then if $M\not\cong G$ then $R(G,[M])$ is empty, but the proof of this required the classification of finite simple groups. \par
We note that if $N$ is any regular subgroup of $B$ then (by basically \cite[Theorem 6.3.2]{Hall1959}) $Norm_B(N)$ is canonically isomorphic to $Hol(N)\cong N\rtimes Aut(N)$. More generally, one can enumerate $R(G,[M])$ by first enumerating
$$
S(M,[G])=\{U\leq Norm_B(M)\ |\ U\text{ regular and }U\cong G\}
$$
where, again, for $M$ a regular subgroup of $B$, $Norm_B(M)\cong Hol(M)$. That is we consider those regular subgroups of $Norm_B(M)\cong Hol(M)$ that are regular and isomorphic to $G$. The correspondence between $|R(G,[M])|$ and $|S([M,[G])|$ is given by the following result due to Byott \cite[p.3220]{Byott1996} which we translate into the terminology we have already established.
\begin{proposition}
For $G$ and $[M]$ as given above
$$
|R(G,[M])|=\dfrac{|Aut(G)|}{|Aut(M)|}|S(M,[G])|
$$
where $Aut(M)$ and $Aut(G)$ are the automorphism groups of $M$ and $G$.
\end{proposition}
This approach has advantages and disadvantages in that, while it doesn't easily yield the element of $R(G,[M])$ from an element of $S(M,[G])$, it does give the counts of one in terms of the other, where the computation of $S(M,[G])$ is feasible at the very least, by brute force using a system such as GAP. We utilize this later on to obtain some of the information in some of the tables we shall give. What is more desirable, typically, is to derive $|R(G,[M])|$ or $|S(M,[G])|$ from first principles. In our analysis, we will take a slightly different tack, by inferring that $R(G,[M])$ is empty in certain circumstances, by utilizing one of the consequences of the existence of a Hopf-Galois structure on a field extension. In the setting of a Hopf-Galois extension $K/k$ with action by a $k$-Hopf algebra $H$, one has:
\begin{theorem}
The correspondence $Fix:\{k-\text{sub-Hopf\ algebras\ of\ H}\}\rightarrow \{subfields\ k\subseteq F\subseteq K\}$ given by 
$$
Fix(H')=\{z\in K\ |\ h(z)=\epsilon(h)z\ \forall h\in H'\}
$$
(where $H'\subseteq H$) is injective and inclusion reversing.
\end{theorem}
From Chase and Sweedler \cite{ChaseSweedler1969}, and extrapolated in Greither-Pareigis, and in \cite[Prop 2.2]{TARP-SES} we have:\par
\begin{proposition}
For a normal extension $K/k$ with $G=Gal(K/k)$ which is Hopf-Galois with respect to the action of $H_N=(K[N])^{G}$ the sub-Hopf algebras of $H_N$ are of the form $H_P=(K[P])^G$ where $P$ is any $G$-invariant subgroup of $N$.
\end{proposition}
And as any intermediate field between $k$ and $K$ corresponds to a subgroup $J\leq G$, where $Fix(H_P)=F=K^J$, one has a modification of the aforementioned Galois correspondence.
The following is basically \cite[Thm. 2.4, Cor. 2.5 and 2.6]{TARP-SES}.\par
\begin{theorem}
The correspondence 
$$
\Psi:\{\text{subgroups}\ of\ N\ \text{normalized by} \lG\}\longrightarrow\{\text{subgroups of}\ G\}
$$
given by
$$\Psi(P)=Orb_{P}(i_G)=\{q(i_G)\ |\ q\in P\}=J$$
is injective and $K^{H_P}=F=K^J$.
\end{theorem}
We note that $J$ {\it is} a subgroup of $G$ and also that $|P|=[K:F]=|J|$. We observe that if $P$ is a characteristic subgroup of $N$ then it is automatically normalized by $\lG$, and, as mentioned above $|\Psi(P)|=|P|$. As such, since $|N|=|G|$ by regularity, if $m\big| |G|$ we let 
\begin{align*}
Sub_m(G)&=\{\text{subgroups of $G$ of order $m$}\}\\
CharSub_m(N)&=\{\text{characteristic subgroups of $N$ of order $m$}\}\\
\end{align*}
and thus we have an injective correspondence 
$$\Psi:CharSub_m(N)\rightarrow Sub_m(G)$$
for each $m\big| |G|$ so that $|CharSub_m(N)|\leq |Sub_m(G)|$.\par
The question we consider is, for a given $N$ where $N\cong M$, can we discern whether  $|CharSub_m(N)|> |Sub_m(G)|$ for at least one $m$, in which case one must conclude that $R(G,[M])=\O$? What is seemingly unlikely about this approach yielding anything is that one expects the class of characteristic subgroups to be somewhat meager, certainly in comparison to the collection of all subgroups. But, for those of a given order $m$ dividing $|G|$ this actually happens relatively often. We start with the first class of examples where this analysis applies. The 5 groups of order 12 are $Q_3 ,C_{12}, A_4,D_6$, and $C_6\times C_2$ and by direct computation we find three pairings $R(G,[M])$ which are empty by this criterion.
\begin{align*}
(G,[M]) &= (A_4,Q_3) \rightarrow &|Sub_6(G)|=0\text{ and }|CharSub_6(M)|=1\\
(G,[M]) &= (A_4,C_{12}) \rightarrow &|Sub_6(G)|=0\text{ and }|CharSub_6(M)|=1\\
(G,[M]) &= (A_4,D_6) \rightarrow &|Sub_6(G)|=0\text{ and }|CharSub_6(M)|=1\\
\end{align*}
which is a modest set of examples, but representative of some basic motifs which we'll explore in more detail presently. Examining the full table of $|R(G,[M])|$ we see where these fit in, and also observe the two other empty pairings.\par\vskip0.125in
\scalebox{0.9}{
\begin{tabular}{|c|c|c|c|c|c|}\hline
$G\downarrow$ \ $M\rightarrow$ & $Q_3$ & $C_{12}$ & $A_4$ & $D_6$ & $C_6\times C_2$ \\ \hline
$Q_3$ & 2& 3& 12& 2& 3\\ \hline
$C_{12}$ & 2& 1& 0& 2& 1\\ \hline
$A_4$ &\cellcolor{Gray} 0& \cellcolor{Gray}0& 10& \cellcolor{Gray}0& 4\\ \hline
$D_6$ & 14& 9& 0& 14& 3\\ \hline
$C_6\times C_2$ & 6& 3& 4& 6& 1\\ \hline
\end{tabular}}\par
We highlight the fact that for $G=A_4$ and $M=Q_3, D_6,\text{ and }C_{12}$ that $|Sub_6(G)|=0$ and $|CharSub_6(M)|=1$.\par
That is, $G$ has no-subgroup of index 2, which is a basic exercise in group theory, and $Q_3, D_6,\text{ and }C_6\times C_2$ have unique (hence characteristic) subgroups of index 2. As it turns out, examples like this are quite common instances of the $|CharSub_m(N)|>|Sub_m(G)|$ condition. 
\section{Index Two Subgroups}
Following Nganou \cite{Nganou} we can apply some basic, yet very useful, group theory facts to examine the index 2 subgroups of a given group.
\begin{theorem}[\cite{Nganou}]
For a finite group $G$, where $n=|G|$, the subgroup $G^2=\langle\{g^2\ |\ g\in G\}\rangle$ is such that
$$
|Sub_{n/2}(G)|=|Sub_{n/2}(G/G^2)|
$$
where, since $[G,G]\subseteq G^2$, $G/G^2$ is an elementary Abelian group of order $2^m$. Moreover, $|Sub_{n/2}(G/G^2)|=2^m-1$ since the index 2 subgroups correspond to hyperplanes in the finite vector space $G/G^2$.\par
\end{theorem}
i.e. $|Sub_{n/2}(G)|=[G:G^2]-1$. As a corollary to this, he also notes:
\begin{corollary}
If $G$ is a finite group then $G$ has no index 2 subgroups iff $[G:G^2]=1$ iff $G$ is generated by squares. And $G$ has a unique index 2 subgroup iff $[G:G^2]=2$.
\end{corollary}
And indeed, $A_4$ has no index 2 subgroups since it is generated by squares since every three cycle is the square of its inverse. There are other examples of even order groups without index 2 subgroups. In degree 24, let $G=SL_2(\mathbb{F}_3)$. There are 15 groups $M$ of order $24$, of which 12 have the property that $|CharSub_{12}(M)|>0$.\par
\noindent If $M=C_{3}\rtimes C_{8},C_{24},S_4,C_{2}\times A_{4}$ then $|Sub_{12}(M)|=1$ so $|CharSub_{12}(M)|=1$.\par
\noindent If $M=C_{3}\rtimes Q_{2},D_{12},C_2\times (C_{3}\rtimes C_{4}),C_{12}\times C_{2},C_{3}\times D_{4}$ then $|Sub_{12}(M)|=3$ and $|CharSub_{12}(M)|=1$.\par
\noindent If $M=C_{4}\times S_{3},(C_{6}\times C_{2})\rtimes C_{2}$ then $|CharSub_{12}(M)|=3$.\par
\noindent If $M=C_{2}\times C_{2}\times S_{3}$ then $|Sub_{12}(M)|=7$ and $|CharSub_{12}(M)|=1$.\par
In fact, there are only 3 non-empty $R(SL_2(\mathbb{F}_3),[M])$, namely $M=SL_2(\mathbb{F}_3)$, $C_3\times Q_2$ and $C_6\times C_2\times C_2$. Of course, not all the cases where the pairing is empty correspond to $M$ having a unique subgroup of index 2. Nonetheless, the number of characteristic subgroups of $M$ of index 2 is larger than the number of index 2 subgroups of $G$.
We present the full table of $|R(G,[M])|$ values, highlighting those determined to be zero via this criterion. As it turns out, all of the cases where $|CharSub_{m}(M)|>|Sub_{m}(G)|$ occur when $m=12$. \par\newpage
\hskip-1.0in\scalebox{0.9}{\begin{tabular}{|c|c|c|c|c|c|c|c|c|}\hline
$G\downarrow$ \ $M\rightarrow$ & $C_{3} \rtimes C_{8}$& $C_{24}$& $SL_{2}(\mathbb{F}_3)$& $C_{3} \rtimes Q_{2}$& $C_{4} \times S_{3}$& $D_{12}$& $C_{2} \times (C_{3} \rtimes C_{4})$ \\ \hline
$C_{3} \rtimes C_{8}$ & 4 & 6 & 24 & 4 & \cellcolor{Gray} 0 & 4 & 0 \\ \hline
$C_{24}$ & 4 & 2 & 0 & 4 & \cellcolor{Gray} 0 & 4 & 0 \\ \hline
$SL_2(\mathbb{F}_3)$ & \cellcolor{Gray} 0 & \cellcolor{Gray} 0 & 10 & \cellcolor{Gray} 0 & \cellcolor{Gray} 0 & \cellcolor{Gray}0 & \cellcolor{Gray} 0 \\ \hline
$C_{3} \rtimes Q_{2}$ & 28 & 18 & 0 & 28 & 56 & 28 & 28 \\ \hline
$C_{4}   \times S_{3}$ & 16 & 12 & 0 & 28 & 56 & 28 & 52\\ \hline
$D_{12}$ & 4 & 6 & 0 & 28 & 56 & 28 & 76\\ \hline
$C_{2}   \times (C_{3} \rtimes C_{4})$ & 24 & 12 & 0 & 28 & 56 & 28 & 36\\ \hline
$(C_{6}   \times C_{2}) \rtimes C_{2}$ & 12 & 6 & 0 & 28 & 56 & 28 & 60\\ \hline
$C_{12}   \times C_{2}$ & 8 & 4 & 0 & 12 & 24 & 12 & 20 \\ \hline
$C_{3}   \times D_{4}$ & 4 & 2 & 0 & 12 & 24 & 12 & 28 \\ \hline
$C_{3}   \times Q_{2}$ & 12 & 6 & 16 & 12 & 24 & 12 & 12\\ \hline
$S_{4}$ & 0 & 0 & 0 & 0 & \cellcolor{Gray} 0 & 0 & 0 \\ \hline
$C_{2}   \times A_{4}$ & 0 & 0 & 0 & 0 & \cellcolor{Gray} 0 & 0 & 0 \\ \hline
$C_{2}   \times C_{2}   \times S_{3}$ & 0 & 0 & 0 & 228 & 456 & 228 & 228\\ \hline
$C_{6}   \times C_{2}  \times  C_{2}$ & 0 & 0 & 0 & 84 & 168 & 84 & 84\\ \hline
\end{tabular}}\par\vskip0.25in
\hskip-1.0in\scalebox{0.9}{\begin{tabular}{|c|c|c|c|c|c|c|c|c|c|}\hline
$G\downarrow$ \ $M\rightarrow$ & $(C_{6} \times C_{2}) \rtimes C_{2}$& $C_{12} \times C_{2}$& $C_{3} \times D_{4}$& $C_{3} \times Q_{2}$& $S_{4}$& $C_{2} \times A_{4}$& $C_{2} \times C_{2} \times S_{3}$& $C_{6} \times C_{2} \times C_{2}$\\ \hline
$C_{3} \rtimes C_{8}$ & \cellcolor{Gray} 0 & 0 & 6 & 6 & 0 & 0 & 0 & 0\\ \hline
$C_{24}$ & \cellcolor{Gray} 0 & 0 & 2 & 2 & 0 & 0 & 0 & 0\\ \hline
$SL_2(\mathbb{F}_3)$ & \cellcolor{Gray} 0 & \cellcolor{Gray} 0 & \cellcolor{Gray} 0 & 8 & \cellcolor{Gray} 0 & \cellcolor{Gray} 0 & \cellcolor{Gray} 0 & 8\\ \hline
$C_{3} \rtimes Q_{2}$ & 56 & 18 & 18 & 6 & 0 & 0 & 28 & 6\\ \hline
$C_{4} \times S_{3}$ & 56 & 30 & 18 & 6 & 24 & 0 & 40 & 12\\ \hline
$D_{12}$ & 56 & 42 & 18 & 6 & 0 & 0 & 52 & 18\\ \hline
$C_{2} \times (C_{3} \rtimes C_{4})$ & 56 & 30 & 18 & 6 & 0 & 48 & 32 & 12\\ \hline
$(C_{6} \times C_{2}) \rtimes C_{2}$ & 56 & 42 & 18 & 6 & 24 & 48 & 44 & 18\\ \hline
$C_{12}\times C_{2}$ & 24 & 10 & 6 & 2 & 0 & 0 & 16 & 4\\ \hline
$C_{3} \times D_{4}$ & 24 & 14 & 6 & 2 & 16 & 0 & 20 & 6\\ \hline
$C_{3}  \times Q_{2}$ & 24 & 6 & 6 & 2 & 0 & 0 & 12 & 2\\ \hline
$S_{4}$ & \cellcolor{Gray} 0 & 0 & 0 & 0 & 8 & 36 & 24 & 48\\ \hline
$C_{2}   \times A_{4}$ & \cellcolor{Gray} 0 & 0 & 0 & 8 & 12 & 16 & 8 & 8\\ \hline
$C_{2}   \times C_{2}  \times  S_{3}$ & 456 & 126 & 126 & 42 & 48 & 0 & 152 & 24\\ \hline
$C_{6}   \times C_{2}   \times C_{2}$ & 168 & 42 & 42 & 14 & 0 & 112 & 56 & 8\\ \hline
\end{tabular}}

\par\newpage

We note that there are total of $76$ different $(G,[M])$ for which $R(G,[M])=\O$, of which this method predicted $20$. As an interesting aside, one {\it can} find extensions $K/\mathbb{Q}$ where $Gal(K/\mathbb{Q})\cong SL_{2}(\mathbb{F}_3)$. 
For example, Heider and Kolvenbach \cite{HeiderKolvenbach}, found that the splitting field of 
$$
f(x)=x^8+9x^6+23x^4+14x^2+1\in\mathbb{Z}[x]
$$
is one such $SL_{2}(\mathbb{F}_3)$ Galois extension.\par
We use the notation 
$$
I_2(G)=[G:G^2]-1
$$
for the number of index 2 subgroups, as given in Crawford and Wallace \cite{CrawfordWallace} who, using Goursat's theorem, present a number of basic facts, namely
\begin{itemize}
\item $I_2(G_1\times G_2)=I_2(G_1)I_2(G_2)+I_2(G_1)+I_2(G_2)$
\item If $I_2(G)>0$ then $I_2(G)\equiv 1\text{, or }3\ (mod\ 6)$
\end{itemize}
Nganou also shows this by observing that $(G_1\times G_2)^2=G_1^2\times G_2^2$ and therefore that $[G_1\times G_2:(G_1\times G_2)^2]=[G_1:G_1^2][G_2:G_2^2]$, and also that if $|G|$ is odd then $I_2(G)=0$ automatically. In actuality, the full machinery of Goursat's theorem, which is used to count subgroups of arbitrary direct products, is not needed since, for subgroups of index 2, and later on index $p$, it's straightforward to enumerate the subgroups via the subgroup indices. Some examples of this were seen in the degree 24 examples earlier, such as
\begin{align*}
I_2(C_{2}\times A_{4})&=I_2(C_{2})I_2(A_{4})+I_2(C_{2})+I_2(A_{4})&=1\cdot 0+0+1=1\\
I_2(C_{12}\times C_{2})&=I_2(C_{12})I_2(C_{2})+I_2(C_{12})+I_2(C_{2})&=1\cdot 1+1+1=3\\
I_2(C_{3}\times D_{4})&=I_2(C_{3})I_2(D_{4})+I_2(C_{3})+I_2(D_{4})&=0\cdot 3+0+3=3\\
I_2(C_{4}\times S_{3})&=I_2(C_{4})I_2(S_{3})+I_2(C_{4})+I_2(S_{3})&=1\cdot 1+1+1=3\\
\end{align*}
What is most interesting about the formula 
$$I_2(G_1\times G_2)=I_2(G_1)I_2(G_2)+I_2(G_1)+I_2(G_2)$$
is that it allows us to readily generate examples of (even order) groups with $0$ or $1$ index two subgroups given that, without loss of generality, $I_2(G_1)=0$ and $I_2(G_2)=0$ or $1$ for then $I_2(G_1\times G_2)=0$ or $1$ as well. If $I_2(G_1)=0$ and $I_2(G_2)=0$ then, of course, $I_2(G_1\times G_2)=0$. If $G_1$ has odd order then $I_2(G_1)=0$ so if either $G_1$ has odd order and $G_2$ even, or both $G_1$ and $G_2$ are even, with $I_2(G_1)=I_2(G_2)=0$ as in the table below, then $I_2(G_1\times G_2)=0$.\par
\begin{itemize}
\item $A_4$
\item $SL_{2}(\mathbb{F}_3)$
\item $(C_{2}\times C_{2})\rtimes C_{9}$
\item $(C_{4}\times C_{4})\rtimes C_{3}$
\item $C_{2}^4\rtimes C_{3}$
\item $C_{2}^3\rtimes C_{7}$
\item $C_{2}^4\rtimes C_{5}$
\item any non-Abelian simple group
\end{itemize}

If $I_2(G_1)=0$ and $I_2(G_2)=1$ then $I_2(G_1\times G_2)=1$.\par
For example:
\begin{multicols}{2}
$G_1$
\begin{itemize}
\item $C_r$ for $r$ odd
\item $A_4$
\item $SL_{2}(\mathbb{F}_3)$
\item $(C_{2}\times C_{2})\rtimes C_{9}$
\item $(C_{4}\times C_{4})\rtimes C_{3}$
\item $C_{2}^4\rtimes C_{3}$
\item $C_{2}^3\rtimes C_{7}$
\item $C_{2}^4\rtimes C_{5}$
\item any non-Abelian simple group
\end{itemize}
\columnbreak
$G_2$
\begin{itemize}
\item $C_s$ for $s$ even
\item $S_n$ for $n\geq 3$
\item $D_{n}$ for $n$ odd
\item $C_3\rtimes C_4$
\item $(C_{3}\times C_{3})\rtimes C_{2}$
\item the non-split extension of $SL_{2}(\mathbb{F}_3)$ by $C_{2}$ (AKA the non-split extension of $C_{2}$ by $S_{4}$)
\end{itemize}
\end{multicols}

The formula for computing $I_2$ of a direct product of two groups can be generalized to a direct product of any number of groups. For example, in degree 36
\begin{align*}
I_2(C_{3}\times C_{3}\times C_{4})&=I_2(C_{3})I_2(C_{3}\times C_{4})+I_2(C_{3})+I_2(C_{3}\times C_{4})\\
                                &=0\cdot 1+0+1\\ 
                                &=1
\end{align*}
which is in agreement with the computation done directly by $[M:M^2]-1$. Note: If we expand out $I_2(G_1\times G_2\times G_3)$ then we find that
\begin{align*}
I_2(G_1\times &G_2\times G_3)\\
=&e_1(I_2(G_1),I_2(G_2),I_3(G_3))+e_2(I_2(G_1),I_2(G_2),I_3(G_3))+\\
&e_3(I_2(G_1),I_2(G_2),I_3(G_3))\\
=&I_2(G_1)+I_2(G_2)+I_2(G_3)+I_2(G_1)I_2(G_2)+I_2(G_1)I_2(G_3)+I_2(G_2)I_2(G_3)+\\
&I_2(G_1)I_2(G_2)I_2(G_3)\\
\end{align*}
Also, it's not hard to prove that this 'product formula' for $I_2(G_1\times G_2)$ holds for semi-direct products of cyclic groups.
\begin{proposition} If $C_r$ and $C_s$ are cyclic groups then 
$$(C_r\rtimes C_s)^2=C_r^2\rtimes C_s^2$$
and therefore that
\begin{align*}
[C_r\rtimes C_s:(C_r\rtimes C_s)^2]&=[C_r:C_r^2][C_s:C_s^2]\\
I_2(C_{r}\rtimes C_{s})&=I_2(C_{r})I_2(C_{s})+I_2(C_{r})+I_2(C_{s})\\
\end{align*}
\end{proposition}
\begin{proof}
If $C_r=\langle x\rangle$ and $C_s=\langle y \rangle$ then
$$
C_r^2=\begin{cases} \langle x\rangle & \text{$r$ odd} \\ 
                   \langle x^2\rangle & \text{$r$ even}\\
\end{cases}
$$
and similarly for $C_s^2$. In either case, $C_r^2$ is characteristic in $C_r$ meaning that $C_r^2\rtimes C_s^2$ {\it is} a subgroup of $C_r\rtimes C_s$ where clearly $C_r^2\rtimes C_s^2\leq (C_r\rtimes C_s)^2$. Now any semi-direct product 
$C_r\rtimes C_s$ arises due to an action of the form $y(x)=x^u$ for $u\in U_{r}$. If $(x^i,y^j)\in C_r\rtimes C_s$ then $(x^i,y^j)^2=(x^{i+u^ji},y^{2j})$, where $y^{2j}$ 
clearly lies in $C_s^2$. The question is whether the first coordinate $x^{i+u^ji}$ lies in $C_r^2$. However, this is easy since if $r$ is even then $u$ must be odd and thus $1+u^j$ is even, which means $i(1+u^j)$ is even. And if $r$ is odd then, as observed above, $C_r=C_r^2$ so that, either way, $x^{i(1+u^j)}\in C_r^2$.
\end{proof}
We have seen examples already in degree 12 and 24,
\begin{align*}
I_2(C_{3}\rtimes C_{4})&=I_2(C_{3})I_2(C_{4})+I_2(C_{3})+I_2(C_{4})=0\cdot 1+0+1=1\\
I_2(C_{3}\rtimes C_{8})&=I_2(C_{3})I_2(C_{8})+I_2(C_{3})+I_2(C_{8})=0\cdot 1+0+1=1\\
\end{align*}
and similarly, we can control $I_2(C_r\rtimes C_s)$ by careful choices of $r$ and $s$, to make it $0$ and/or $1$. If we define 
\begin{align*}
z_2(n)&=\text{the number of groups of order $n$ with no index two subgroups}\\
u_2(n)&=\text{the number of groups of order $n$ with one index two subgroup}\\
\end{align*}
then we have empty pairings $R(G,[M])$ corresponding to $z_2(n)*u_2(n)$ for $n\leq 256$.\par\newpage
\begin{tabular}{|c|c|c|c|c|}\hline
 $n$ & $z_2$ & $u_2$ & $z_2*u_2$ & $(\text{\# of groups of order $n$})^2$ \\ \hline
12 & 1 & 2 & 2 & 25\\ \hline
24 & 1 & 4 & 4 & 225\\ \hline 
36 & 2 & 6 & 12 & 196\\ \hline
48 & 2 & 8 & 16 & 2704\\ \hline
56 & 1 & 2 & 2 & 169\\ \hline
60 & 2 & 6 & 12 & 169\\ \hline
72 & 2 & 13 & 26 & 2500\\ \hline
80 & 1 & 3 & 3 & 2704\\ \hline
84 & 2 & 6 & 12 & 225\\ \hline
96 & 3 & 15 & 45 & 53361\\ \hline
108 & 7 & 18 & 126 & 2025\\ \hline
120 & 2 & 12 & 24 & 2209\\ \hline
132 & 1 & 4 & 4 & 100\\ \hline
144 & 5 & 25 & 125 & 38809\\ \hline
156 & 2 & 9 & 18 & 324\\ \hline
160 & 1 & 5 & 5 & 56644\\ \hline
168 & 5 & 12 & 60 & 3249\\ \hline
180 & 3 & 18 & 54 & 1369\\ \hline
192 & 9 & 39 & 351 & 2380849\\ \hline
204 & 1 & 6 & 6 & 144\\ \hline
216 & 8 & 45 & 360 & 31329\\ \hline
228 & 2 & 6 & 12 & 225\\ \hline
240 & 4 & 26 & 104 & 43264\\ \hline
252 & 5 & 18 & 90 & 2116\\ \hline
\end{tabular}\par
\section{Non Index Two Examples}
Even though index 2 subgroups are a convenient source of 'counterexamples' to the condition $|Sub_m(G)|\geq|CharSub_m(N)|$ condition, there are of course many other possible subgroup orders where our method applies. For example, $R(A_5,[C_5 \times A_4])=\O$ which is known already by Byott's result since $A_5$ is simple, of course, but can be inferred by our method because $C_5\times A_4$ has a unique subgroup of order $20$ since $A_4$ has a unique subgroup of order $4$. Another example is $((C_5 \times C_5) \rtimes C_3,[C_{75}])$ since $|Sub_{15}(C_{75})|=1$ of course, and $|CharSub_{15}((C_5 \times C_5) \rtimes C_3)|=0$ since any subgroup of order $15$ in $(C_5 \times C_5) \rtimes C_3$ would have to be cyclic and intersect the $(C_{5}\times C_{5})\cong \mathbb{F}_{5}^2$ component in a subgroup of order $5$. And since no automorphism of $C_5\times C_5$ of order $3$ could arise due to scalar multiplication, this aforementioned subgroup of order $5$ would not be normal in $C_5\times C_5$. One other example we can consider is the case of $R(S_5,[C_{120}])$. Of course, $C_{120}$ has one subgroup of order $15$, but since any group of order $15$ is cyclic, then it's clear that $S_5$ has no such subgroup.\par
Index two or not, using GAP, \cite{GAP4} one can readily enumerate the subgroups, both characteristic and otherwise, of each group of a given low order. We present a table of some compiled counts of the number of pairs $R(G,[M])$ which are forced to be empty because $|Sub_m(G)|<|CharSub_m(M)|$ for some $m$, which we denote $|Z|$, as compared with square of the number of groups of order $n$, denoted $|R|^2$, representing all possible pairings of groups of order $n$. We also should point out that, if the criterion applied, it frequently happened in index 2.\par\newpage
\begin{multicols*}{3}
\begin{tabular}{|c|c|c|}\hline
 $n$ & $|Z|$ & $|R|^2$ \\ \hline
1 & 0 & 1 \\ \hline
2 & 0 & 1 \\ \hline
3 & 0 & 1 \\ \hline
4 & 0 & 4 \\ \hline
5 & 0 & 1 \\ \hline
6 & 0 & 4 \\ \hline
7 & 0 & 1 \\ \hline
8 & 0 & 25 \\ \hline
9 & 0 & 4 \\ \hline
10 & 0 & 4 \\ \hline
11 & 0 & 1 \\ \hline
12 & 3 & 25 \\ \hline
13 & 0 & 1 \\ \hline
14 & 0 & 4 \\ \hline
15 & 0 & 1 \\ \hline
16 & 5 & 196 \\ \hline
17 & 0 & 1 \\ \hline
18 & 2 & 25 \\ \hline
19 & 0 & 1 \\ \hline
20 & 0 & 25 \\ \hline
21 & 0 & 4 \\ \hline
22 & 0 & 4 \\ \hline
23 & 0 & 1 \\ \hline
24 & 20 & 225 \\ \hline
25 & 0 & 4 \\ \hline
26 & 0 & 4 \\ \hline
27 & 0 & 25 \\ \hline
28 & 0 & 16 \\ \hline
29 & 0 & 1 \\ \hline
30 & 0 & 16 \\ \hline
31 & 0 & 1 \\ \hline
32 & 38 & 2601 \\ \hline
\end{tabular}
\vfill\null
\begin{tabular}{|c|c|c|}\hline
 $n$ & $|Z|$ & $|R|^2$ \\ \hline
33 & 0 & 1 \\ \hline
34 & 0 & 4 \\ \hline
35 & 0 & 1 \\ \hline
36 & 34 & 196 \\ \hline
37 & 0 & 1 \\ \hline
38 & 0 & 4 \\ \hline
39 & 0 & 4 \\ \hline
40 & 11 & 196 \\ \hline
41 & 0 & 1 \\ \hline
42 & 0 & 36 \\ \hline
43 & 0 & 1 \\ \hline
44 & 0 & 16 \\ \hline
45 & 0 & 4 \\ \hline
46 & 0 & 4 \\ \hline
47 & 0 & 1 \\ \hline
48 & 244 & 2704 \\ \hline
49 & 0 & 4 \\ \hline
50 & 2 & 25 \\ \hline
51 & 0 & 1 \\ \hline
52 & 0 & 25 \\ \hline
53 & 0 & 1 \\ \hline
54 & 8 & 225 \\ \hline
55 & 0 & 4 \\ \hline
56 & 15 & 169 \\ \hline
57 & 0 & 4 \\ \hline
58 & 0 & 4 \\ \hline
59 & 0 & 1 \\ \hline
60 & 28 & 169 \\ \hline
61 & 0 & 1 \\ \hline
62 & 0 & 4 \\ \hline
63 & 0 & 16 \\ \hline
64 & 1576 & 71289 \\ \hline
\end{tabular}
\vfill\null
\begin{tabular}{|c|c|c|}\hline
 $n$ & $|Z|$ & $|R|^2$ \\ \hline
65 & 0 & 1 \\ \hline
66 & 0 & 16 \\ \hline
67 & 0 & 1 \\ \hline
68 & 0 & 25 \\ \hline
69 & 0 & 1 \\ \hline
70 & 0 & 16 \\ \hline
71 & 0 & 1 \\ \hline
72 & 422 & 2500 \\ \hline
73 & 0 & 1 \\ \hline
74 & 0 & 4 \\ \hline
75 & 1 & 9 \\ \hline
76 & 0 & 16 \\ \hline
77 & 0 & 1 \\ \hline
78 & 0 & 36 \\ \hline
79 & 0 & 1 \\ \hline
80 & 149 & 2704 \\ \hline
81 & 5 & 225 \\ \hline
82 & 0 & 4 \\ \hline
83 & 0 & 1 \\ \hline
84 & 28 & 225 \\ \hline
85 & 0 & 1 \\ \hline
86 & 0 & 4 \\ \hline
87 & 0 & 1 \\ \hline
88 & 4 & 144 \\ \hline
89 & 0 & 1 \\ \hline
90 & 8 & 100 \\ \hline
91 & 0 & 1 \\ \hline
92 & 0 & 16 \\ \hline
93 & 0 & 4 \\ \hline
94 & 0 & 4 \\ \hline
95 & 0 & 1 \\ \hline
96 & 4197 & 53361 \\ \hline
\end{tabular}
\end{multicols*}\par\newpage
\begin{multicols*}{3}
\begin{tabular}{|c|c|c|}\hline
 $n$ & $|Z|$ & $|R|^2$ \\ \hline
97 & 0 & 1 \\ \hline
98 & 2 & 25 \\ \hline
99 & 0 & 4 \\ \hline
100 & 20 & 256 \\ \hline
101 & 0 & 1 \\ \hline
102 & 0 & 16 \\ \hline
103 & 0 & 1 \\ \hline
104 & 11 & 196 \\ \hline
105 & 0 & 4 \\ \hline
106 & 0 & 4 \\ \hline
107 & 0 & 1 \\ \hline
108 & 327 & 2025 \\ \hline
109 & 0 & 1 \\ \hline
110 & 0 & 36 \\ \hline
111 & 0 & 4 \\ \hline
112 & 92 & 1849 \\ \hline
113 & 0 & 1 \\ \hline
114 & 0 & 36 \\ \hline
115 & 0 & 1 \\ \hline
116 & 0 & 25 \\ \hline
117 & 0 & 16 \\ \hline
118 & 0 & 4 \\ \hline
119 & 0 & 1 \\ \hline
120 & 350 & 2209 \\ \hline
121 & 0 & 4 \\ \hline
122 & 0 & 4 \\ \hline
123 & 0 & 1 \\ \hline
124 & 0 & 16 \\ \hline
125 & 0 & 25 \\ \hline
126 & 24 & 256 \\ \hline
127 & 0 & 1 \\ \hline
128 & 481816 & 5419584 \\ \hline
\end{tabular}
\vfill\null\hskip0.25in
\begin{tabular}{|c|c|c|}\hline
 $n$ & $|Z|$ & $|R|^2$ \\ \hline
129 & 0 & 4 \\ \hline
130 & 0 & 16 \\ \hline
131 & 0 & 1 \\ \hline
132 & 12 & 100 \\ \hline
133 & 0 & 1 \\ \hline
134 & 0 & 4 \\ \hline
135 & 0 & 25 \\ \hline
136 & 14 & 225 \\ \hline
137 & 0 & 1 \\ \hline
138 & 0 & 16 \\ \hline
139 & 0 & 1 \\ \hline
140 & 6 & 121 \\ \hline
141 & 0 & 1 \\ \hline
142 & 0 & 4 \\ \hline
143 & 0 & 1 \\ \hline
144 & 6790 & 38809 \\ \hline
145 & 0 & 1 \\ \hline
146 & 0 & 4 \\ \hline
147 & 2 & 36 \\ \hline
148 & 0 & 25 \\ \hline
149 & 0 & 1 \\ \hline
150 & 26 & 169 \\ \hline
151 & 0 & 1 \\ \hline
152 & 4 & 144 \\ \hline
153 & 0 & 4 \\ \hline
154 &        0 & 16 \\ \hline
155 &        0 & 4 \\ \hline
156 &       37 & 324 \\ \hline
157 &        0 & 1 \\ \hline
158 &        0 & 4 \\ \hline
159 &        0 & 1 \\ \hline
160 &     3145 & 56644 \\ \hline
\end{tabular}
\vfill\null\hskip0.125in
\begin{tabular}{|c|c|c|}\hline
 $n$ & $|Z|$ & $|R|^2$ \\ \hline
161 &        0 & 1 \\ \hline
162 &       70 & 3025 \\ \hline
163 &        0 & 1 \\ \hline
164 &        0 & 25 \\ \hline
165 &        0 & 4 \\ \hline
166 &        0 & 4 \\ \hline
167 &        0 & 1 \\ \hline
168 &      448 & 3249 \\ \hline
169 &        0 & 4 \\ \hline
170 &        0 & 16 \\ \hline
171 &        0 & 25 \\ \hline
172 &        0 & 16 \\ \hline
173 &        0 & 1 \\ \hline
174 &        0 & 16 \\ \hline
175 &        0 & 4 \\ \hline
176 &       54 & 1764 \\ \hline
177 &        0 & 1 \\ \hline
178 &        0 & 4 \\ \hline
179 &        0 & 1 \\ \hline
180 &      276 & 1369 \\ \hline
181 &        0 & 1 \\ \hline
182 &        0 & 16 \\ \hline
183 &        0 & 4 \\ \hline
184 &        4 & 144 \\ \hline
185 &        0 & 1 \\ \hline
186 &        0 & 36 \\ \hline
187 &        0 & 1 \\ \hline
188 &        0 & 16 \\ \hline
189 &        6 & 169 \\ \hline
190 &        0 & 16 \\ \hline
191 &        0 & 1 \\ \hline
192 &        219139 & 2380849 \\ \hline
\end{tabular}
\end{multicols*}
\par\newpage
\section{$R(C_{p^n},[A])$ Revisited}
Lastly, we consider an already solved problem! For $G=C_{p^n}$, for each $p^r|p^n$ one has, of course, $|Sub_{p^r}(G)|=1$.\par
For a non-cyclic Abelian $p$-group $M$ of order $p^n$, one has that $M\cong C_{p^{\lambda_1}} \times C_{p^{\lambda_2}}\cdots\times C_{p^{\lambda_t}}$ where $\lambda_1+\lambda_2+\cdots+\lambda_t=n$ is a partition, where, without loss of generality, $\lambda_1\leq \lambda_2\leq\cdots\leq \lambda_{t}$. Not unexpectedly, a given non-cyclic Abelian $p$-group has {\it many} subgroups for each order. 
Tarnauceanu and Toth, \cite{TarnauceanuToth}, aggregate a number of older results as:
\begin{theorem}
For every partition $\mu\preceq\lambda$ (i.e. $\mu_i\leq \lambda_i$) the number of subgroups of type $\mu$ in $G_{\lambda}$ is 
$$
\alpha_{\lambda}(\mu;p)=\prod_{i\geq 1}p^{(a_i-b_i)b_{i+1}}{a_i-b_{i+1}\choose b_i-b_{i+1}}_p,
$$
where $\lambda'=(a_1,\dots)$ and $\mu'=(b_1,\dots)$ are the partitions conjugate to $\lambda$ and $\mu$, respectively, and 
$$
{n\choose k}_p=\dfrac{\prod_{i=1}^{n}(p^i-1)}{\prod_{i=1}^{k}(p^i-1)\prod_{i=1}^{n-k}(p^i-1)}
$$
is the Gaussian binomial coefficient (it is understood that $\prod_{i=1}^{m}(p^i-1)=1$ for $m=0$).
\end{theorem}
In \cite{KerbyTurner} Kerby and Rode  (extending an old result due to Reinhold Baer) show that the characteristic subgroups of $M$ of order $p^r$ correspond to partitions/tuples of $r$, $\bold{a}=\{a_i\}$ termed 'canonical', namely
\begin{itemize}
\item $a_i\leq a_{i+1}$ for all $i\in\{2,\dots,t\}$ and
\item $a_{i+1}-a_{i}\leq \lambda_{i+1}-\lambda_{i}$ for all $i\in\{1,\dots,t-1\}$
\end{itemize}
where, the total number of subgroups of order $r$ would be the total number of such partitions for each $r$ from $1$ to $n$.\par
What one discovers is that for sufficiently large $n$ there are various $r\leq n$ for which there are more than one canonical partitions of $r$. For example, if $M=C_{p}\times C_{p^3}$ ($n=4$) there are two canonical partitions of $2$, namely $\{1,1\}$ and $\{0,2\}$, which therefore correspond to two characteristic subgroups of order $p^2$. As such $R(C_{p^4},[C_{p}\times C_{p^3}])=\O$. Another example is for $M=C_{p}\times C_{p^4}$, where there are two characteristic subgroups of order $p^2$ and two of order $p^3$.\par
For $n=6$ we have four different partitions of $n$ which each give rise to more than one canonical tuples for subgroups of particular orders, namely $6=1+2+3=1+1+4=2+4=1+5$, and thus
\begin{itemize}
\item $R(C_{p^6},[C_{p}\times C_{p^2}\times C_{p^3}])=\O$
\item $R(C_{p^6},[C_{p}\times C_{p}\times C_{p^4}])=\O$
\item $R(C_{p^6},[C_{p^2}\times C_{p^4}])=\O$
\item $R(C_{p^6},[C_{p}\times C_{p^5}])=\O$
\end{itemize}
Looking at larger $n$, we can consider all the partitions of $n$, which we denote $np$ and then count those which give rise to {\it more than one} canonical tuples for some $r\leq n$, which we denote $nc$. One observes that the fraction $nc/np$ of partitions of $n$ which give rise to $>1$ characteristic subgroups of some order approaches 1.\par
\begin{multicols}{2}
\begin{tabular}{|c|c|c|c|}\hline
 $n$ & $nc$ & $np$ & $nc/np$ \\ \hline
 1&  0& 1& 0 \\ \hline
 2&  0& 2& 0 \\ \hline
 3&  0& 3& 0 \\ \hline
 4&  1& 5& 0.2 \\ \hline
 5&  1& 7& 0.142 \\ \hline
 6&  4& 11& 0.363 \\ \hline
 7&  4& 15& 0.266 \\ \hline
 8&  10& 22& 0.454\\ \hline
 9&  13& 30& 0.433\\ \hline
 10&  23& 42& 0.547\\ \hline
 11&  27& 56& 0.482\\ \hline
 12&  52& 77& 0.675\\ \hline
 13&  60& 101& 0.594\\ \hline
\end{tabular}\hskip0.5in
\columnbreak
\begin{tabular}{|c|c|c|c|}\hline
 $n$ & $nc$ & $np$ & $nc/np$ \\ \hline
 14&  94& 135& 0.696\\ \hline
 15&  118& 176& 0.670\\ \hline
 16&  175& 231& 0.757\\ \hline
 17&  213& 297& 0.717\\ \hline
 18&  310& 385& 0.805\\ \hline
 19&  373& 490& 0.761\\ \hline
 20&  528& 627& 0.842\\ \hline
 21&  643& 792& 0.811\\ \hline
 22&  862& 1002& 0.860\\ \hline
 23&  1044&1255& 0.832\\ \hline
 24&  1403 & 1575 & 0.891\\ \hline
 25 &  1699& 1958 & 0.868\\ \hline
 26 & 2199 & 2436 & 0.903\\ \hline

\end{tabular}\par
\end{multicols}
The takeaway from this is that we should expect $R(C_{p^n},[M])$ to be empty for most non-cyclic Abelian $p$-groups. Of course, this is not a new result, but it's interesting to compare this method to the usual argument which relies on the impossibility of $G\leq Hol(N)$ if $G$ is cyclic of order $p^n$ and $N$ is a non-cyclic $p$-group of the same order.
\bibliography{charsub}
\bibliographystyle{plain}
\end{document}